\documentclass[11pt,fleqn]{article}
\usepackage{amsmath,epic,amssymb}
\usepackage{amsthm,amsfonts,amscd,epsfig}
\usepackage[dvipsnames]{xcolor}
\usepackage{authblk}
\usepackage{tikz}
\usetikzlibrary{arrows.meta}
\usetikzlibrary{shapes.geometric}
\usepackage{pgfplots}
\usepgfplotslibrary{polar}
\usepgflibrary{shapes.geometric}
\usetikzlibrary{calc}
\pgfplotsset{my style/.append style={axis x line=middle, axis y line=middle, xlabel={$\alpha$}, ylabel={$\beta$}, axis equal }}
\pgfplotsset{every axis legend/.append style={
at={(1.02,1)},anchor= north west}}
\pgfplotsset{samples=100}
\usepgfplotslibrary{fillbetween}

\newtheorem{theorem}{Theorem}[section]
\newtheorem{lemma}[theorem]{Lemma}
\newtheorem{remark}[theorem]{Remark}
\newtheorem{proposition}[theorem]{Proposition}
\newtheorem{corollary}[theorem]{Corollary}
\newtheorem{definition}[theorem]{Definition}
\newtheorem{example}[theorem]{Example}
\newcount\refno
\newcommand\be{\begin{equation}}
\newcommand\ee{\end{equation}}
\newcommand\bn{\begin{eqnarray}}
\newcommand\en{\end{eqnarray}}
\newcommand\bns{\begin{eqnarray*}}
\newcommand\ens{\end{eqnarray*}}
\newcommand\bd{\begin{definition}}
\newcommand\ed{\end{definition}}
\newcommand\br{\begin{remark}}
\newcommand\er{\end{remark}}
\newcommand\bt{\begin{theorem}}
\newcommand\et{\end{theorem}}
\newcommand\bp{\begin{proposition}}
\newcommand\ep{\end{proposition}}
\newcommand\bc{\begin{corollary}}
\newcommand\ec{\end{corollary}}
\newcommand\bl{\begin{lemma}}
\newcommand\el{\end{lemma}}
\newcommand\pf{\begin{proof}}

\newcommand\bR{{\mathbb R}}
\newcommand\bN{{\mathbb N}}

\newcommand\cR{{\cal R}}

\newcommand{\F}{\mbox{$\mathcal{F}$}}

\makeindex

\allowdisplaybreaks

\begin{document}

\title{Total Positivity of Almost-Riordan Arrays}

\author[1]{Tian-Xiao He}
\author[2]{Roksana S\l{}owik}
\affil[1]{Department of Mathematics, Illinois Wesleyan University, Bloomington, IL 61702-2900, Indiana, USA}
\affil[2]{Faculty of Applied Mathematics, Silesian University of Technology, Kaszubska 23, 44-100 Gliwice, Poland}

\pagestyle{myheadings} 
\markboth{T. X. He and R. S\l{}owik}{TP Almost-Riordan Arrays}

\maketitle

\begin{abstract}
\noindent 
In this paper we study the total positivity of almost-Riordan arrays $(d(t)|\, g(t), f(t))$ and establish its necessary conditions and sufficient conditions, particularly, for some well used formal power series $d(t)$. We present a semidirect product of an almost-array and use it to transfer a total positivity problem for an almost-Riordan array to the total positivity problem for a quasi-Riordan array. We find the sequence characterization of total positivity of the almost-Riordan arrays. The production matrix $J$ of an almost-Riordan array $(d|\, g,f)$ is presented so that $J$ is totally positive implies the total positivity of both the almost-Riordan array $(d|\, g,f)$ and the Riordan array $(g,f)$. We also present a counterexample to illustrate that this sufficient condition is not necessary.  If the production matrix $J$ is tridiagonal, then the expressions of its principal minors are given. By using expressions, we find a sufficient and necessary condition of the total positivity of almost-Riordan arrays with tridiagonal production matrices. A numerous examples are given to demonstrate our results.


\end{abstract}

\section{Introduction}\label{Sec1}
Following Karlin \cite{Kar} and Pinkus \cite{Pin}, an infinite matrix is called totally positive (abbreviately, TP), if its minors of all orders are nonnegative. An infinite nonnegative sequence $(a_n)_{n\geq 0}$ is called a P\'olya frequency sequence (abbreviately, PF), if its Toeplitz matrix

\[
\left[a_{i-j}\right]_{i,j\geq 0}=\left[ \begin{array} {lllll} a_0 & & & &  \\
a_1& a_0 & & & \\ a_2 &a_1& a_0& &  \\ a_3& a_2 & a_1& a_0 & \\
\vdots& \vdots &\vdots&\vdots & \ddots \end{array}\right]
\]
is TP. We say that a finite sequence $(a_0, a_1,\ldots, a_n)$ is PF if the corresponding infinite sequence $(a_0, a_1, \ldots, a_n, 0, \ldots)$ is PF. Denote by ${\bN}$ the set of all nonnegative integers. A fundamental characterization for PF sequences is given by Schoenberg et al.\cite{AESW, ASW, Kar}, which states that a sequence $(a_n)_{n\geq 0}$ is PF if and only if its generating function can be written as 

\begin{equation}\label{0}
\sum_{n\geq 0} a_n t^n=C t^ke^{\gamma t}\frac{\Pi_{j\geq 0} (1+\alpha_j t)}{\Pi_{j\geq 0} (1-\beta_j t)},
\end{equation}
where $C>0$, $k\in {\bN}$, $\alpha_j$, $\beta_j$, $\gamma \geq 0$, and $\sum_{j\geq 0}(\alpha_j+\beta_j)<\infty$. In this case, the above generating function is called a P\'olya frequency formal power series. For some relevant results, see, for example, Brenti \cite{Bre} and Pinkus \cite{Pin}. 

Let us denote by $\F_k$ the set of all formal power series of the form $\sum_{n=k}^\infty a_nt^n$ with $a_k\neq 0$. 

\begin{definition}\label{def:1.3}\cite{Barry16}
Let $d, g\in \F_0$ with $d(0), g(0)=1$ and $f\in \F_1$ with $f'(0)=1$. We call the following matrix an almost-Riordan array with respect to $d,g,$ and $f$ and denote it by $(d|\,g,f)$:

\begin{equation}\label{1.13}
(d|\,g,f)=(d, tg, tgf, tgf^2,\cdots),
\end{equation}
where $d$, $tg$, $tgf$, $t^2gf\cdots$, are the generating functions of the $0$th, $1$st, $2$nd, $3$rd, $\cdots$, columns of the matrix $(d|\,g,f)$, respectively. 
\end{definition}

It is clear that $(d|\,g,f)$ can be written as 
\begin{equation}\label{1.14}
(d|\,g,f)=\left( \begin{matrix} d(0) & 0\\ (d-d(0))/t & (g,f) \end{matrix}\right),
\end{equation}
where $(g,f)=(g,gf, gf^2,\ldots)$, a Riordan array. Particularly, if $d=g$ and $f=t$, then the almost-Riordan array $(d|\,g,f)$ reduces to the Appell-type Riordan array $(g, t)$. 

Barry, Pantelidis, and one of the authors \cite{BHP} present the following analogy for the almost-Riordan arrays.

\begin{theorem}\label{thm:2.3}\cite{BHP}
The set of all almost-Riordan arrays defined by \eqref{1.13} forms a group, denoted by $a{\cR}$, with respect to the multiplication defined by 
\begin{equation}\label{1.15}
(a|\,g,f)(b|\,d,h)=\left( \left. a+\frac{tg}{f}(b(f)-1)\right |\,gd(f), h(f)\right),
\end{equation}
where $a, g, b, d\in \F_0$ with $a(0), g(0), b(0), d(0)=1$ and $f, h\in \F_1$ with $f'(0)=1$ and $h'(0)=1$, respectively, and $(a|\,g,f)$ and $(b|\,d,h)$ are the almost-Riordan arrays defined by \eqref{1.13} or \eqref{1.14}.
\end{theorem}

Let us introduce one more group that is closely related to the groups of Riordan and almost-Riordan arrays.

\begin{definition}\label{def:2.3}\cite{He} 
Let $g\in \F_0$ with $g(0)=1$ and $f\in \F_1$. We call the following matrix a quasi-Riordan array and denote it by $[g,f]$.
\begin{equation}\label{2.6}
[g,f]:=(g,f,tf,t^2f,\ldots),
\end{equation}
where $g$, $f$, $tf$, $t^2f\cdots$ are the generating functions of the $0$th, $1$st, $2$nd, $3$rd, $\cdots$, columns of the matrix $[g,f]$, respectively. It is clear that $[g,f]$ can be written as 
\begin{equation}\label{2.8}
[g,f]=\left( \begin{matrix} g(0) & 0\\ (g-g(0))/t & (f/t,t)\end{matrix}\right),
\end{equation}
where $(f,t)=(f,tf,t^2f,t^3f,\ldots)$ and $(f/t, t)$ is an Appell Riordan array. 
Particularly, if $f=tg$, then the quasi-Riordan array $[g,tg]=(g, t)$, a Appell-type Riordan array. 

Clearly, $[g,f]=(g|f/t, t)$. 

Let $A$ and $B$ be $m\times m$ and $n\times n$ matrices, respectively. Then we define the direct sum of $A$ and $B$ by

\begin{equation}\label{2.5} 
A\oplus B =\left[ \begin{matrix} A &0 \\ 0 &B\end{matrix}\right]_{(m+n)\times(m+n)}.
\end{equation}

In this notation the Riordan array $(g,f)$ satisfies (see for instance \cite{MMW}) 

\begin{equation}\label{2.9}
(g,f)=[g,f]([1]\oplus (g,f)).
\end{equation}
\end{definition}

Denote by $q{\cR}$ the set of all quasi-Riordan arrays defined by \eqref{2.6}. In \cite{He} it is shown that $q{\cR}$ is a group with respect to regular matrix multiplication. More precisely, there is the following result.

\begin{theorem}\label{thm:1.2}\cite{He}
The set of all quasi-Riordan arrays $q{\cR}$ is a group, called the quasi-Riordan group, with respect to the multiplication represented in 
\begin{equation}\label{1.11}
[g,f][d,h]=\left[g+\frac{f}{t}(d-1), \frac{fh}{t}\right],
\end{equation}
which is derived from the first fundamental theorem for quasi-Riordan arrays (FFTQRA),

\begin{equation}\label{1.12}
[g,f]u=gu(0)+\frac{f}{t}(u-u(0)).
\end{equation}
Hence, $[1,t]$ is the identity of $q{\cR}$. 
\end{theorem}

In \cite{BHP}, it has been shown that the quasi-Riordan group is a normal subgroup of the almost-Riordan group. In authors' paper \cite{HS}, we discussed the total positivity of quasi-Riordan arrays. For instance, let $(g(t), f(t))$ be a Riordan array, where $g(t)=\sum_{n\geq 0} g_nt^n$ and $f(t)=\sum_{n\geq 1}f_nt^n$. If the lower triangular matrix 

\begin{equation}\label{eq:1}
\begin{array}{rl}
Q & =[g,f]=(g|f/t,t)=
\left [ \begin{array}{llllll} g_0& 0& 0& 0& 0& \cdots\\
g_1& f_1 & 0& 0& 0&\cdots\\
g_2 &f_2& f_1& 0& 0& \cdots\\
g_3& f_3& f_2& f_1&0&\cdots\\
\vdots &\vdots& \vdots& \vdots&\vdots&\ddots\end{array}\right]
\\
& =\left( g(t), f(t), tf(t), t^{2}f(t),\ldots\right)
\\
\end{array}
\end{equation}
is totally positive (TP), then so is $R=(g,f)$ (cf. \cite{MMW} or \cite{He, HS}). 
Other interesting criteria for total positivity of Riordan arrays can be found in \cite{CLW,CW}. 

In this paper we are going to focus on total positivity of almost-Riordan arrays and its connections with total positivity of Riordan and almost-Riordan arrays.

According to the definition of the almost-Riordan array, if an almost-Riordan array $(d|\,g,f)$ is TP, then the corresponding Riordan array $(g,f)$ shown in \eqref{1.14} must be TP. In other words, a necessary condition for the total positivity of an almost- Riordan array is that its corresponding Riordan array is TP. Therefore, the discussion of total positivity covers the discussion of total positivity of Riordan arrays, noting that research on the latter is still ongoing and attracts many researchers. However, in this paper we are interested in two questions regarding the total positivity of almost-Riordan arrays. The first is the conditions for finding $d$ in an almost Riordan array $(d|\, g,f)$ that has a corresponding TP Riordan array $(g,f)$ such that $(d|\, g ,f)$ is TP. The second is to determine the sequence characteristics of the TP almost-Riordan array. Let $(d|\, g,f)$ be the almost-Riordan array of $d_0=1$, then the generating function of its $A$-, $Z$- and $W$- sequences is given by the following formula (see \cite{AK})

\begin{align}
&A(t)=\frac{t}{\overline{f}(t)},\label{eq:A-Z_general-1}\\
&Z(t)=\frac{t(g(\overline{f}(t))-z_0d(\overline{f}(t))}{\overline{f}(t)g(\overline{f}(t)))}+z_0, \label{eq:A-Z_general-2}\\
&W(t)=\frac{t(d(\overline{f}(t))-d_0-w_0\overline{f}(t)d(\overline{f}(t)))}{(\overline{f}(t))^2g(\overline{f}(t))}+w_0,\label{eq:A-Z_general-3}
\end{align}
where $z_0=g_0$ and $w_0=d_1$. 

In particular, for the quasi-Riordan array $[d,g]=(d|\, g/t,t)$, the generating functions of its $A$-, $Z$-, and $W$-sequences are reduced to 

\begin{align}
&A(t)=1,\label{eq:A_Z_W-1}\\
&Z(t)=\frac{g(t)-z_0td(t)}{g(t)}+z_0, \label{eq:A_Z_W-2}\\ 
&W(t)=\frac{(1-w_0t)d(t)-1}{g(t)}+w_0,\label{eq:A_Z_W-3},
\end{align}

For instance, for quasi-Riordan array $[1/(1-t), t/(1-t)]$, $A(t)=Z(t)=W(t)=1$. 

The first problem will be discussed in Section $2$ is about the total positivity of almost-Riordan arrays $(d|\,g,f)$ for some special power series $d$ when the Riordan array $(g,f)$ is TP. More precisely, we give a necessary condition for the total positivity of an almost-Riordan array $(d|\, g,f)$ and sufficient conditions for the the total positivity of an almost-Riordan arrays $(d|\, g,f)$ when $d(t)=1, tg(t)+\alpha$, and $d_0+d_1t$, $\alpha, d_0, d_1\in {\bR}$, respectively. We also present a semidirect product of an almost-array and use it to transfer a total positivity problem of an almost-Riordan array to the total positivity problem of a quasi-Riordan array.  In Section \ref{Sec3}, we study the sequence characterization of TP almost-Riordan arrays. The production matrix $J$ of an almost-Riordan array $(d|\, g,f)$ is presented and the following general result is claimed: The $J$ is TP implies the total positivity of both the almost-Riordan array $(d|\, g,f)$ and the Riordan array $(g,f)$. We also use a counterexample to illustrate that $J$ being TP does not necessarily mean that its corresponding almost-Riordan array is TP.  If the production matrix $J$ is tridiagonal, then the expressions of its principal minors can be obtained. By using the expressions, we find a sufficient and necessary condition of the total positivity of almost-Riordan arrays with tridiagonal production matrices. In the remaining of Section \ref{Sec3}, numerous example are given to demonstrate our results.

\section{Total positivity of certain almost-Riordan arrays}\label{Sec2}

Let us consider two special cases for $d$ in a Riordan array and discuss its TP and the TP for the corresponding almost-Riordan array. 

Based on the discussion in the Introduction, we have

\begin{proposition}\label{pro:2.0}
If the almost-Riordan array $(d|\,g,f)$ is $TP$, then the Riordan array $(g,f)$ is $TP$. 
\end{proposition} 

If $d(t)=1$, then from \eqref{2.8} we have 

\[
(d|\,g,f)=\left( \begin{matrix} 1 & 0\\ 0& (g,f)\end{matrix}\right). 
\]

Hence, one can observe

\begin{proposition}\label{pro:2.2}
The almost-Riordan array $(1|\,g,f)$ is $TP$ if and only if the Riordan array $(g,f)$ is $TP$. 
\end{proposition}

In particular, if $g$ and $f$ are P\'olya frequency power series, then the matrix $(g,f)$ is TP, which implies that $(1|\,g, f)$ is TP \cite{CW}.  

The following example shows there exists a non-TP almost-Riordan array $(d|\, g,f)$, in which $d$, $g$ and $f$ are P\'olya frequency.

\begin{example}
\label{ex:gf_not_TP}
Let $d(t)=(1+t)^2$, $g(t)=1/(1-t)$, and $f(t)=t$. Then $d(t)$, $g(t)$ and $f(t)$ are P\'olya frequency, and 
$$(g, f)=
\begin{bmatrix}
1&&&&&\\
1&1&&&&\\
1&1&1&&&\\
1&1&1&1&&\\
1&1&1&1&1&\\
\vdots&&&&&\ddots\\
\end{bmatrix}
$$
is TP. However, we have
$$(d|\,g,f)=[(1+t)^2, t/(1-t)]=
\begin{bmatrix}
1&&&&&\\
2&1&&&&\\
1&1&1&&&\\
0&1&1&1&&\\
0&&1&1&1&\\
\vdots&&&&&\ddots\\
\end{bmatrix}$$ 

whose minor 

\[
M^{1,2,3}_{0,1,2}=\det\left( \begin{array}{lll} 2&1&0\\1& 1&1\\0&1&1\end{array}\right)=-1<0.
\]
Hence, $(d|\,g,f)$ is not TP although $(g,f)$ is TP and $d$ is P\'olya frequency. 
\end{example}

The following example shows there exists a class of TP almost-Riordan arrays 
$(d|\, g,f)$, which $d=1+d_1t$ is P\'olya frequency, $g(t)= f(t)/t$,  
and $f(t)=f_1t+f_2t^2$. Since  

\[
(d|\, f/t,f)=[d,f]=\left[\begin{array}{llll} 1& 0 & & \\ d_1 & f_1 &0 & \\ & f_2& f_1 & \\ & & f_2 &\\  &   \ddots &\ddots &\ddots\end{array}
\right].
\]

Hence, from Proposition 2.6 of \cite[Proposition 2.6]{CLW},  
$(d|\, f/t,t)$is TP if and only if $f_1\geq 0$, which implies that $[1+d_1t, f_1t+f_2t^2]$ is TP if $d_1,f_1,f_2\geq 0$. 

From the above examples, we may see quasi-Riordan arrays as the elements of the normal subgroup of almost-Riordan group make important rule in our discussion.

As we have seen, the most interesting question here is whether having fixed $(g,f)$ can we choose $d$ in such way to ensure total positivity of $(d|g,f)$ and what, in such case, is the connection between $d$ and the pair $(g,f)$.  

For every Riordan array $(g,f)$ there exists $d$ such that the almost- Riordan array $(d|g,f)$ is TP. In fact, we have 

\begin{theorem}\label{thm:2.4}
Suppose $g\in\F_0$, $f\in\F_1$ such that the Riordan array $(g,f)$ is TP. Then the almost-Riordan array $\left(t\cdot g+\alpha|\,g,f\right)$ is also TP for every $\alpha>0$.
\end{theorem}

\begin{proof}
Observe that the Riordan array $\left(t\cdot g+\alpha|\,g,f\right)=R$ can be written in the partitioned form
\[
R=\left[
\begin{array}{c|c}
\alpha & 0 \\
\hline 
g & (g,f) \\
\end{array}
\right].
\]
We now consider its minors.
\begin{enumerate}
\item 
If a minor of $R$ does not contain the column $0$, then it is simply a minor of $(g,f)$, so it is nonnegative.
\item 
Suppose now that a minor of $R$ contains the column $0$.
\begin{enumerate}
\item 
If it contains the row $0$, then expanding along it, we get a minor of $(g,f)$ multiplied by $\alpha$ which is a nonnegative number.
\item 
If this minor does not contain the row $0$, then it is a minor of the partitioned matrix 
\[
\tilde R=\left[\begin{array}{c|c|c}
g_0 & g_0 & 0 \\
\hline
\frac{g-g_0}t & \frac{g-g_0}t & \left(\frac{gf}t,f\right)\\
\end{array}\right].
\]
\begin{enumerate}
\item 
Now, if minor contains both the two first columns of $\tilde R$, then it is equal to $0$.  
\item 
If the minor contains one or none of the first two columns of $\tilde R$, then it is simply one of the minors of $(g,f)$, so it is nonnegative. 
\end{enumerate}
\end{enumerate}
\end{enumerate}
\end{proof}

\par 
Moreover, for every TP array $(g,f)$ there exist $d$ that does not depend on $g$ and $f$ such that $(d|g,f)$ is TP.

\begin{theorem}\label{thm:2.4}
Let $g\in\F_0$, $f\in\F_1$ and let $d=d_0+d_1t$ with $d_0,d_1\geqslant 0$. If the Riordan array $(g,f)$ is TP, then $(d_0+d_1t|\,g,f)$ is also TP.
\end{theorem}

\begin{proof}
Let 
\[
R=(d_0+d_1t|\,g,f)=\left[
\begin{array}{c|c}
d_0 & 0 \\
\hline 
d_1 &  \\
0 & (g,f) \\
\vdots &  \\
\end{array}
\right].
\]
Consider the minors of $R$.
\begin{enumerate}
\item If a minor does not contain the column $0$, then it is a minor of $(g,f)$ that, by the assumption, is nonnegative. 
\item Suppose now that a minor contains the column $0$. 
\begin{enumerate}
\item 
If it contains the row $0$, then, after expansion along this row, one gets a minor of $(g,f)$ multiplied by $d_0$ which is a nonnegative number.
\item 
If it does not contain the row $0$, but it contains the first row, then it is a minor of 
\[
\tilde R=\left[\begin{array}{c|c|c}
d_1 & g_0 & 0 \\
\hline
0 & \frac{g-g_0}t & \left(\frac{gf}t,f\right)\\
\end{array}\right].
\]
Thus, expanding now along the column $0$, we get a minor of $(g,f)$ multiplied by $d_1$ that is nonnegative.
\item 
If it does not contain neither the zeroth nor the first row, then it is a minor of 
\[
\tilde R=\left[\begin{array}{c|c|c}
0 & \frac{g-g_0}t & \left(\frac{gf}t,f\right)\\
\end{array}\right]
\]
and it is equal to $0$.
\end{enumerate}
\end{enumerate}
\end{proof}

It is known that every Riordan array $(g,f)$ can be written as the semidirect product 
\[
(g,f)=(g,t)(1,f).
\]

Consequently, we have the following result. 

\begin{proposition}\label{thm:2.5}
Every almost-Riordan array $(d|\,g,f)$ can be written as the semidirect product 

\begin{equation}\label{eq:R_factoriz-1}
(d|\, g,f)=[d,tg](1|\,1,f),
\end{equation}
or equivalently, 

\begin{equation}
\label{eq:R_factoriz}
(d|\,g,f)=\left[\begin{array}{c|c}
d_0 & 0  \\
\hline 
\frac{d-d_0}{t}&  (g,t)\\
\end{array}\right]
\left[\begin{array}{c|c}
1 & 0\\
\hline 
0 & (1,f)\\
\end{array}\right].
\end{equation}
\end{proposition}

\begin{proof}
From the factorization of Riordan arrays shown above, we may prove the theorem.
\end{proof}

Clearly, the total positivity of the two matrices from the above decomposition ensures the total positivity of $R$. Suppose that both $g$ and $f$ are P\'olya frequency formal power series. Then, from \cite{CW}, we know that $(g,f)$ is TP. Thus, to obtain $d$ from $(\ref{eq:R_factoriz})$ such that $(d|\,g,f)$ is TP, we only need to focus on $d$. 

Summing up, we get the following 

\begin{corollary}
Let $d,g\in\F_0$, $f\in\F_1$. If $f$ is a P\'olya frequency formal power series and the quasi-Riordan array 
$[d, tg]$ is TP, then the almost-Riordan array $(d|\,g,f)$ is also TP.
\end{corollary}

Thus, if $f$ is an arbitrary P\'olya frequency formal power series, then in order to construct totally positive  almost-Riordan array $(d|\,g,f)$ it suffices to discuss the total positivity of the quasi-Riordan array $[d,tg]$.   

\begin{example}
Let us consider $g(t)=\frac 1{1-\alpha t}$ and  $d(t)=\displaystyle\sum_{n=0}^\infty d_nt^n$ in $[d,tg]$, where $tg(t)$ is a P\'olya frequency formal power series. The quasi-Riordan array $[d,tg]$ is a special case of an almost-Riordan array. 
From \cite{HS} we know that if its production matrix 
\[
J=\left[ \begin{array}{llllll} w_0& z_0& & & &  \\
w_1& z_1& 1 & & & \\ 
w_2& z_2& 0 & 1 & & \\
w_3& z_3& 0 & 0 & 1 &\\
\vdots &\vdots &\vdots &\vdots& \vdots &\ddots
\end{array}\right]
\]
is TP, where $Z(t)=\sum_{n\geq 0} z_n t^n$ and $W(t)=\sum_{n\geq 0} w_nt^n$ are represented in \eqref{eq:A_Z_W-2} and \eqref{eq:A_Z_W-3}, 
then $[d,g]$ is also TP. 

Moreover, also from \cite[Theorem 3.4]{HS}, we know that $J$ is TP if and only if $Z(t)=z_0+z_1t$, $W(t)=w_0+w_1t$, and $w_0z_1-w_1z_0\geqslant 0$. For the given $g=1/(1-\alpha t)$, from \eqref{eq:A-Z_general-2} and \eqref{eq:A-Z_general-3}, we have 

\begin{align*}
&Z(t)=z_0+z_1t=\frac{\frac{t}{1-\alpha t}-z_0td(t)}{\frac{t}{1-\alpha t}}+z_0,\quad \textrm{and} \\
&W(t)=w_0+w_1t=\frac{(1-w_0t)d(t)-1}{\frac{t}{1-\alpha t}}+w_0.
\end{align*}

Assume $z_0=1$. The above system has solution 
\[d(t)=\frac{1-z_1t}{1-\alpha t}\]
if
 
\[
w_0=\alpha -z_1 \quad\textrm{and}\quad w_1=z_1(\alpha -z_1).
\] 

Denote $z_1=\beta$, those conditions come down to 

\[
Z(t)=1+\beta t,\,\, W(t)=\alpha-\beta+(\alpha\beta-\beta^2)t, \,\, \mbox{and}\,\, d(t)=\frac{1-\beta t}{1-\alpha t}
\]

Since $w_0,w_1\geq 0$, we obtain $0\leq\beta\leq \alpha$. Note that the necessary and sufficient condition of the total positivity of quasi-Riordan array shown in \cite[Theorem 3.4]{HS} is satisfied, i.e. 
\[
w_0z_1-w_1z_0=w_0\beta-w_1=(\alpha-\beta)\beta-(\alpha\beta-\beta^2)=0\,\,\textrm{for all } \alpha,\beta.
\]

Consequently,  for any $\beta\in\left[0,\alpha\right]$ and any P\'olya frequency formal power series $f$, the almost-Riordan array $\left(\frac{1-\beta t}{1-\alpha t}|\frac 1{1-\alpha t},f(t)\right)$ is TP.
\end{example}

Similarly as in the previous example, if we replace $g(t)$ to be $\alpha$, then $J$ is TP only if $w_0= z_1= w_1=0$, and $\alpha=d(t)=1$ when $z_0=1$.

\section{Sequence Characterization of the TP Almost-Riordan Arrays}\label{Sec3}

We now discuss the sequence characterization of the TP almost-Riordan arrays. We start from a general case of the characteristic sequence with the generating functions shown in \eqref{eq:A-Z_general-1}-\eqref{eq:A-Z_general-3}.

\subsection{General result}

We call the following matrix the {\it production matrix} of the almost-Riordan matrix $[g,f]$:

\begin{equation}\label{eq:2}
J=\left[ \begin{array}{llllll} w_0& z_0& & & &  \\
w_1& z_1& a_0& & & \\ w_2& z_2& a_1& a_0 & & \\
w_3& z_3& a_2& a_1& a_0 &\\
\vdots &\vdots &\vdots &\vdots& \vdots &\ddots
\end{array}\right]
\end{equation} 

It can be checked that 

\begin{equation}\label{eq:3}
(d|g,f)J=\overline{(d|g,f)},
\end{equation}
where $\overline{(d|\,g,f)}$ is the matrix by deleting the first row of $(d|\,g,f)$.

\begin{proposition}\label{pro:3.1}
Let $(d|\,g,f)$ and $J$ be two matrices defined by \eqref{2.6} and \eqref{eq:2}, respectively. If $J$ is TP, then so are $(d|\,g,f)$ and $(g,f)$. 
\end{proposition}

\begin{proof}
It can be proved by using \eqref{eq:3} and mathematical induction.
\end{proof}

\subsection{The tridiagonal case}

Let us consider some special form of $J$. 
As from Thm.4.3 \cite{Pin} we know that a Jacobi matrix

\[
\left[\begin{matrix}
a_1 & b_1 & & &  & \\
c_1 & a_2 & b_2 & & & \\
 & c_2 & a_3 & b_3 & & \\
 & & \ddots & \ddots & \ddots & \\
 & & & & c_{n-1} & a_n\\
\end{matrix}\right]
\]
is TP if and only if its all elements and its all principal minors containing consecutive rows and columns are nonnegative, we conclude that in order to check total positivity of production matrix of the tridiagonal form 

\[
J=\left[\begin{matrix}
w_0 & z_0 & & & & & \\
w_1 & z_1 & a_0 & & & & \\
 & z_2 & a_1 & a_0 & & & \\
 & & a_2 & a_1 & a_0 & & \\
 & & & a_2 & a_1 & a_0 & \\
 & & & & \ddots & \ddots & \ddots\\
\end{matrix}\right]
\]
it suffices to check the principal minors. 
Principal minor $J_n$ of degree $n$ of the above $J$ is equal to 
\begin{equation}
\label{eq:det(J_n)}
\det(J_n)=\begin{cases}
w_0 & \textrm{for }n=1\\
w_0z_1-z_0w_1 & \textrm{for }n=2\\
w_0z_1a_1-w_0z_2a_0-w_1z_0a_1 & \textrm{for }n=3\\
(w_0z_1-w_1z_0)\det(T_{n-2})-w_0z_2a_0\det(T_{n-3}) & \textrm{for }n>3,\\
\end{cases}
\end{equation}
where $T_n$ stands for the $n\times n$ Toeplitz matrix 

\begin{equation}
\label{eq:T_n}
T_n=\left[\begin{matrix}
a_1 & a_0 & & & & \\
a_2 & a_1 & a_0 & & & \\
 & a_2 & a_1 & a_0 & & \\
 & & \ddots & \ddots & \ddots & \\
 & & & a_2 & a_1 & a_0 \\
 & & & & a_2 & a_1 \\
\end{matrix}\right],
\end{equation}

whose determinant can be found by using its expansion along the first column: 

\begin{equation}
\label{eq:det(T_n)}
\begin{array}{rl}
\det(T_n) & =a_1\det(T_{n-1})-a_0a_2\det(T_{n-2})
\\
& =
\begin{cases}
\frac 1{\sqrt{a_1^2-4a_0a_2}}\left[\left(\frac{a_1+\sqrt{a_1^2-4a_0a_2}}{2}\right)^{n+1}-
\left(\frac{a_1-\sqrt{a_1^2-4a_0a_2}}{2}\right)^{n+1}\right] 
\\
\hspace{6cm}\textrm{ if } a_1^2-4a_0a_2\neq 0\\
(n+1)\left(\frac{a_1}2\right)^n \\
\hspace{6cm}\textrm{ if } a_1^2-4a_0a_2=0,\\
\end{cases}
\\
\end{array}
\end{equation}
where the numbers $\frac{a_1+\sqrt{a_1^2-4a_0a_2}}{2}$, $\frac{a_1-\sqrt{a_1^2-4a_0a_2}}{2}$, and $\frac{a_1}2$ are, in fact, the roots of the characteristic equation

\[
x^2-a_1x+a_0a_2=0
\]
of the determinant of $T_n$. We assume here that $\sqrt{a_1^2-4a_0a_2}$ is the principal root of $a_1^2-4a_0a_2$. Note that, from the recurrence relation, it follows that the determinant given by $(\ref{eq:det(T_n)})$ is always real.

Let us discuss first the case when the characteristic equation has a single root. 

\begin{theorem}
\label{thm:one_root}
Suppose that $a_0,a_1,a_2,z_0,z_1,z_2,w_0,w_1\geqslant 0$ and $a_1^2=4a_0a_2$. 
If the numbers
\[
w_0z_1-w_1z_0, \quad \mbox{and} \quad w_0z_1a_1-w_1z_0a_1-2w_0z_2a_0
\]
are nonnegative, then the almost-Riordan array $(d|g,f)$ defined by its sequence characterization shown in Eqs. \eqref{eq:A-Z_general-1}-\eqref{eq:A-Z_general-3} is TP.
\end{theorem}

\begin{proof}
By the assumptions $a_1^2=4a_0a_2$, $a_1\geqslant 0$, and $(\ref{eq:det(T_n)})$, we get that $\det(T_n)\geqslant 0$. To show $det(J_n)\geqslant 0$, by $(\ref{eq:det(J_n)})$, it suffices to check the signs of the expressions 
\[
\begin{array}{l}
w_0z_1-z_0w_1,\quad 
w_0z_1a_1-w_0z_2a_0-w_1z_0a_1=(w_0z_1-w_1z_0)a_1-w_0z_2a_0,
\\
(w_0z_1-w_1z_0)(n-1)\left(\frac{a_1}2\right)^{n-2}-w_0z_2a_0(n-2)\left(\frac{a_1}2\right)^{n-3}\textrm{ for } n>3.
\\
\end{array}
\]
The last expression is nonnegative due to 
\[
\frac{a_1(w_0z_1-w_1z_0)}{2w_0z_2a_0}\geqslant\frac{n-2}{n-1}\quad\textrm{for all } n>3, 
\]
and $\sup_{n>3}\left\{\frac{n-2}{n-1}\right\}=1$. Therefore,  
we need to have  
\[
w_0z_1-z_0w_1, 
(w_0z_1-z_0w_1)a_1-w_0a_0z_2,
(w_0z_1-z_0w_1)a_1-2w_0a_0z_2\geqslant 0,
\] 
which are guaranteed by the conditions given by the theorem and the following inequality:
\[
(w_0z_1-w_1z_0)a_1-w_0 z_2a_0\geqslant (w_0z_1-w_1z_0)a_1-2w_0z_2a_0.
\]
\end{proof}

Case when the characteristic equation has two roots, requires considering two more cases. First, consider the situation of the two real roots. 

\begin{theorem}
\label{thm:two_roots_R}
Suppose that $a_0,a_1,a_2,z_0,z_1,z_2,w_0,w_1\geqslant 0$ and $a_1^2-4a_0a_2>0$. 
If the numbers 
\[
w_0z_1-z_0w_1, \quad 
w_0z_1a_1-w_0z_2a_0-w_1z_0a_1 
\]
are nonnegative, then the almost-Riordan array $(d|\,g,f)$ defined by its sequence characterization shown in (\ref{eq:A-Z_general-1})-(\ref{eq:A-Z_general-3}) is TP.
\end{theorem}

\begin{proof}
Clearly, the two numbers from the claim come from Eq.$(\ref{eq:det(J_n)})$, whereas nonnegativity of $a_0$, $a_1$, $a_2$ and positivity of $a_1^2-4a_0a_2$ ensure the positivity of roots of the characteristic equation and consequently of the difference $\left(\frac{a_1+\sqrt{a_1^2-4a_0a_2}}{2}\right)^{n+1}-\left(\frac{a_1-\sqrt{a_1^2-4a_0a_2}}{2}\right)^{n+1}$. 
\end{proof}

For two complex roots we have the following. 

\begin{theorem}
\label{thm:det(T_n)<0}
Let $a_0,a_1,a_2,z_0,z_1,z_2,w_0,w_1\geqslant 0$. If $a_1^2-4a_0a_2<0$, then there exists $n\in\mathbb N$ such that $\det(T_n)<0$.
\end{theorem}

\begin{proof}
Let us shortly write $r$ for the modulus of $\frac{a_1+i\sqrt{4a_0a_2-a_1^2}}2$ (note that since $a_1^2-4a_0a_2<0$, we can treat $\sqrt{4a_0a_2-a_1^2}$ as the arithmetic root) and $\theta$ for its argument. Since $4a_0a_2-a_1^2>0$, $\theta\in\left(0,\frac\pi 2\right]$. 
With this notation we have 
\[
\begin{array}{rl}
\det(T_n) & =\frac 1{i\sqrt{4a_0a_2-a_1^2}}\cdot r^{n+1}\left(e^{(n+1)i\theta}-e^{-(n+1)i\theta}\right)
\\
 & =\frac{2r^{n+1}}{\sqrt{4a_0a_2-a_1^2}}\sin\left((n+1)\theta\right).
\\
\end{array}
\] 

If $\theta=\frac\pi 2$, then for $n=2$, we obtain $\sin\frac{3\pi}2=-1$ in the above expression, so $\det(T_2)<0$. 

If $\theta<\frac\pi 2$, then there exists a number $n$, $n\geqslant 2$, such that $\frac\pi{n+1}<\theta<\frac{2\pi}{n+1}$. Then 
$\pi<(n+1)\theta<2\pi$, so the corresponding sine function value in the above expression is negative, and so is $\det(T_{n+1})$. 
\end{proof}

\begin{example}
As Theorem \ref{thm:det(T_n)<0} indicates, the negativity of $\det(T_n)$ is only a ``matter of time" in the sense that the sign of $\det(T_n)$ depends on the index. For instance, choosing $a_0=2$, $a_1=3$, $a_2=5$, we have $\arg(3+\sqrt{31}i)\approx 0.343\pi$, and $\det(T_2)=-4$. Yet, choosing $a_0=a_2=3$, $a_1=5$, we have $\arg(5+\sqrt{11}i)\approx 0.186\pi$. In this case $\det(T_2)$, $\det(T_3)$, $\det(T_4)$ are positive and only $\det(T_5)$ is the first one to be negative.  
\end{example}

Theorems \ref{thm:two_roots_R} and \ref{thm:det(T_n)<0} sum up in the following. 

\begin{theorem}\label{thm:3.4}
Suppose that $a_0,a_1,a_2,z_0,z_1,z_2,w_0,w_1\geqslant 0$ and $a_1^2-4a_0a_2\neq 0$. The production matrix $J$ is TP if and only if 
\[
a_1^2-4a_0a_2>0, 
\quad 
w_0z_1-w_1z_0,\quad 
w_0z_1a_1-w_0z_2a_0-w_1z_0a_1\geqslant 0.  
\]
\end{theorem}

From (\ref{eq:A-Z_general-1})-\eqref{eq:A-Z_general-3} we have

\begin{align}
&d(t)=\frac{d_0(1-z_1t-z_2tf(t))}{F(t)}\label{0+1}\\
&g(t)=\frac{d_0z_0}{F(t)}\label{0+2}\\
&f(t)=\frac{1-a_1t-\sqrt{(a_1^2-4a_2a_0)t^2-2a_1t+1}}{2a_2t},
\end{align}
where 

\begin{equation}\label{0+4}
F(t)=1-(w_0+z_1) t+(w_0z_1-w_1z_0)t^2-z_2(1-w_0t) tf(t).
\end{equation}

The above formulas can be used to find out the TP almost-Riordan array $(d|\, g,f)$ from the characteristic sequences of a TP almost-Riordan array.

\subsection{Examples of TP almost-Riordan arrays}

Using the results presented at the start of this section, we can provide some examples of TP almost-Riordan arrays. 

\begin{example}\label{ex:delta=0_e1}
Defining the family of the sequences $A$, $Z$, $W$ as 
\[
A(t)=1+\alpha t+\frac{\alpha^2}4t^2,
\quad 
Z(t)=1+t+t^2,
\quad 
W(t)=1+\beta t,
\]
and denoting the family of the sequences $A$, $Z$, and $W$ by $AZW_1(\alpha, \beta)$, from Theorem \ref{thm:3.4} we conclude that the production matrix $J$ (\ref{eq:2}) is TP if and only if 
\[
\alpha, \beta\geqslant 0,\quad 
\alpha(1-\beta)\geqslant 2,
\]
what is depicted on Fig.\ref{pic:ex1}.

\begin{figure}[h]
\caption{Figure of Ex.\ref{ex:delta=0_e1}}
\label{pic:ex1}
\begin{center}
\begin{tikzpicture}
\begin{axis}[my style, width=8cm, height=4cm, xmin=0, xmax=9, ymin=-0.5, ymax=1.5, xtick={2}, ytick={1}, legend style={empty legend}]
\addplot [name path = A, color=red, domain=2:9] {1-2/x};
\addlegendentry{\textcolor{red}{$\beta=1-\frac 2\alpha$}};
\addplot [name path = B, color=black, domain=2:9, forget plot] {0};
\addplot [fill opacity=0.1, color=red] fill between [of = A and B, soft clip={domain=2:9}];
\end{axis}
\end{tikzpicture}
\end{center}
\end{figure}

 Consequently, by Proposition \ref{pro:3.1} and above calculations, we obtain the almost-Riordan arrays with their characterized sequences in the family 
\[
\{ AZW_1(\alpha, \beta):\alpha, \beta \geqslant 0, \beta \leqslant 1-2/\alpha\}.
\]

For instance, for $(\alpha,\beta)=(2,0)$, that is 
\[
A(t)=1+2t+t^2,
\quad 
Z(t)=1+t+t^2,
\quad 
W(t)=1,
\]
Assuming that $d_0=1$, we obtain TP almost-Riordan array 
\[
R=\left(d(t)|g(t),h(t)\right)=
\left[\begin{array}{ccccccc}
1&&&&&&\\
1&1&&&&&\\
1&2&1&&&&\\
1&4&4&1&&&\\
1&9&13&6&1&&\\
1&23&41&26&8&1&\\
\vdots&&&&&&\ddots\\
\end{array}\right]
\]
where from \eqref{0+1}-\eqref{0+4} we have $F(t)=(1-t)(1-t-tf(t))$ and 
\[
\begin{array}{l}
d(t)=\frac 1{1-t},\\
g(t)=\frac{d(t)}{1-t-t f(t)}=\frac 2{(1-t)(1+\sqrt{1-4t})},\\
f(t)=\frac{1-2t-\sqrt{1-4t}}{2t}.\\
\end{array}
\]

For $(\alpha,\beta)=\left(4,\frac 13\right)$ we obtain another TP almot-Riordan array 
\[
R=\left(d(t)|g(t),h(t)\right)=
\left[\begin{array}{ccccccc}
1&&&&&&\\
1&1&&&&&\\
\frac 43&2&1&&&&\\
2&\frac{13}3&6&1&&&\\
\frac{31}9&\frac{37}3&\frac{97}3&10&1&&\\
\frac{68}9&\frac{433}9&\frac{545}3&\frac{229}3&14&1&\\
\vdots&&&&&&\ddots\\
\end{array}\right]
\]
with 
\[
\begin{array}{l}
d(t)=\frac{3\left[1-t-tf(t)\right]}{3-6t+2t^2+(3t^2-3t)f(t)}=\frac{3\left(7-4t+\sqrt{1-8t}\right)}{24-48t+16t^2+(3t-3)\left(1-4t-\sqrt{1-8t}\right)},\\
g(t)=\frac 3{3-6t+2t^2+(3t^2-3t)f(t)}=\frac{24}{24-48t+16t^2+(3t-3)\left(1-4t-\sqrt{1-8t}\right)},\\
f(t)=\frac{1-4t-\sqrt{1-8t}}{8t}.\\
\end{array}
\]
\end{example}

\begin{example}\label{ex:delta=0_e2}
Putting now 
\[
A(t)=1+\alpha t+\frac{\alpha^2}4t^2,
\quad 
Z(t)=1+t+\beta t^2,
\quad 
W(t)=1+\frac 1 2t,
\]
denoting the family of the above sequences by $AZW_2(\alpha, \beta)$, and using Theorem \ref{thm:3.4} we derive this time that the corresponding production matrix $J$ (\ref{eq:2}) is TP if and only if
\[
0\leqslant\beta\leqslant\frac\alpha 4.
\]

\begin{figure}[h]
\caption{Figure of Ex.\ref{ex:delta=0_e2}}
\label{pic:ex2}
\begin{center}
\begin{tikzpicture}
\begin{axis}[my style, width=8cm, height=4cm, xmin=0, xmax=8, ymin=-0.5, ymax=2.5, xtick={2}, ytick={1}, legend style={empty legend}]
\addplot [name path = A, color=red, domain=0:9] {x/4};
\addlegendentry{\textcolor{red}{$\beta=\frac\alpha 4$}};
\addplot [name path = B, color=black, domain=0:9, forget plot] {0};
\addplot [fill opacity=0.1, color=red] fill between [of = A and B, soft clip={domain=0:9}];
\end{axis}
\end{tikzpicture}
\end{center}
\end{figure}

We denote the family of the characterized sequences by $\{ AZW_2(\alpha, \beta): 0\leq \beta\leq \alpha/4\}$, which feasible region is shown in Figure \ref{pic:ex2}. Hence, the almost-Riordan arrays with their characterized sequences in this family, for instance, $(\alpha, \beta)=(1, 1/4)$, $(1/2, 1/16)$, etc., are TP. 

For instance for $(\alpha,\beta)=\left(1,\frac 14\right)$, we obtain TP almost-Riordan array 
\[
R=\left(d(t)|g(t),h(t)\right)=
\left[\begin{array}{ccccccc}
1&&&&&&\\
1&1&&&&&\\
\frac 32&2&1&&&&\\
\frac 52&\frac{15}4&3&1&&&\\
\frac{35}8&7&7&4&1&&\\
\frac{63}8&\frac{105}8&15&\frac{45}4&5&1&\\
\vdots&&&&&&\\
\end{array}\right]
\]
with 
\[
\begin{array}{l}
d(t)=\frac{1-t+\sqrt{1-2t}}{1-2t+(1-t)\sqrt{1-2t}},\\
g(t)=\frac{2}{1-2t+(1-t)\sqrt{1-2t}},\\
f(t)=\frac{2(1-t-\sqrt{1-2t})}t.\\
\end{array}
\]
Another example can be 
\[
R=\left[
\begin{array}{ccccccc}
1&&&&&&\\
1&1&&&&&\\
\frac 32&2&1&&&&\\
\frac 52&\frac{57}{16}&\frac 52&1&&&\\
\frac{137}{32}&\frac{199}{32}&\frac{39}8&3&1&&\\
\frac{473}{64}&\frac{1383}{128}&\frac{283}{32}&\frac{103}{16}&\frac 72&1&\\
\vdots&&&&&&\ddots\\
\end{array}
\right]
\]
for $(\alpha,\beta)=\left(\frac 12,\frac 1{16}\right)$, with
\[
\begin{array}{l}
d(t)=\frac{3t-2-2\sqrt{1-t}}{-t^2+5t-2+2(t-1)\sqrt{1-t}},\\
g(t)=\frac{4}{2(1-t)\sqrt{1-t}+t^2-5t+2},\\
f(t)=\frac{8-4t-8\sqrt{1-t}}{t}.\\
\end{array}
\]
\end{example}

\begin{example}
\label{ex:delta>0_e1}
Let 
\[
A(t)=1+\alpha t+t^2,
\quad 
Z(t)=1+t+t^2,
\quad 
W(t)=1+\beta t,
\]
with $\alpha> 2$. Denote by $AZW_3(\alpha, \beta)$ the family of the sequences $A$, $Z$, $W$ defined above. Then, by Theorem \ref{thm:3.4},  the production matrix $J$ (\ref{eq:2}) is TP if and only if  
\[
 0\leqslant \beta\leqslant 1-\frac 1 \alpha. 
\]

For instance, if $(\alpha, \beta)\in \{AZW_3( \alpha, \beta):\alpha > 2, 0\leqslant \beta\leqslant 1-\frac 1 \alpha\}$ (see its feasible region in Figure \ref{pic:ex3}), say $(\alpha, \beta)=(2, 0)->(3,0)$ and $(4,1/3)$, the corresponding almost-Riordan arrays with the characteristic sequences $AZW_3(2,0)$ and $AZW_3(4, 1/3)$ are TP.

\begin{figure}[h]
\caption{Figure of Ex.\ref{ex:delta>0_e1}}
\label{pic:ex3}
\begin{center}
\begin{tikzpicture}
\begin{axis}[my style, width=8cm, height=4cm, xmin=0, xmax=8, ymin=-0.5, ymax=1.5, xtick={2}, ytick={1}, legend style={empty legend}]
\addplot [name path = A, color=red, domain=2:8] {1-1/x};
\addlegendentry{{\color {red} $\beta=1-\frac 1\alpha$}};
\addplot [name path = B, color=black, domain=2:9, forget plot] {0};
\addplot [fill opacity=0.1, color=red, forget plot] fill between [of = A and B, soft clip={domain=2:8}];
\end{axis}
\end{tikzpicture}
\end{center}
\end{figure}

For $(\alpha,\beta)=(3,0)$ we obtain TP almost-Riordan array 
\[
R=\left(d(t)|g(t),f(t)\right)=\left[\begin{array}{ccccccc}
1&&&&&&\\
1&1&&&&&\\
1&2&1&&&&\\
1&4&5&1&&&\\
1&10&20&8&1&&\\
1&421&78&45&11&1&\\
\vdots&&&&&&\ddots\\
\end{array}\right]
\]
with 
\[
\begin{array}{l}
d(t)=\frac 1{1-t},\\
g(t)=\frac 2{(1-t)(1+t+\sqrt{5t^2-6t+1})},\\
f(t)=\frac{1-3t-\sqrt{5t^2-6t+1}}{2t},\\
\end{array}
\]
whereas for $(\alpha,\beta)=\left(4,\frac 13\right)$
\[
R=\left(d(t)|g(t),f(t)\right)=\left[\begin{array}{ccccccc}
1&&&&&&\\
1&1&&&&&\\
\frac 43&3&1&&&&\\
2&\frac{28}3&7&1&&&\\
\frac{31}9&32&\frac{115}3&11&1&&\\
\frac{68}9&\frac{1090}9&\frac{589}3&\frac{250}3&15&1&\\
\vdots&&&&&&\ddots\\
\end{array}\right]
\]
where 
\[
\begin{array}{l}
d(t)=\frac{-6t-3-3\sqrt{12t^2-8t+1}}{8t^2-3t-3+(3t-3)\sqrt{12t^2-8t+1}},\\
g(t)=\frac 6{3(1-t)\sqrt{12t^2-8t+1}-8t^2-3t+3},\\
f(t)=\frac{1-4t-\sqrt{12t^2-8t+1}}{2t}.\\
\end{array}
\]
\end{example}

\begin{example}\label{ex:delta>0_e2}
Let this time 
\[
A(t)=1+\alpha t+t^2,
\quad 
Z(t)=1+t+\beta t^2,
\quad 
W(t)=1+\frac 13t,
\]
with  $2<\alpha\leqslant 3$. Denote by $AZW_4(\alpha, \beta)$ the family of the sequences defined above. By Theorem \ref{thm:3.4}, the production matrix $J$ (\ref{eq:2}) is TP if and only if
\[
0\leqslant \beta \leqslant 1-\frac \alpha 3.
\]
For instance, if $(\alpha, \beta)\in \{AZW_4( \alpha, \beta):3\geq \alpha > 2, 0\leq 0\leq \beta\leq \frac 1 2-\frac \alpha 6\}$ (see its feasible region in Figure \ref{pic:ex4}). 

\begin{figure}[h]
\caption{Figure of Ex.\ref{ex:delta>0_e2}}
\label{pic:ex4}
\begin{center}
\begin{tikzpicture}
\begin{axis}[my style, width=8cm, height=4cm, xmin=0, xmax=4, ymin=-0.5, ymax=1.5, xtick={2}, ytick={1}, legend style={empty legend}]
\addplot [name path = A, color=red, domain=2:3] {1-x/3};
\addlegendentry{\textcolor{red}{$\beta=1-\frac \alpha 3$}};
\addplot [name path = B, color=black, domain=2:3, forget plot] {0};
\addplot [fill opacity=0.1, color=red, forget plot] fill between [of = A and B, soft clip={domain=2:3}];
\end{axis}
\end{tikzpicture}
\end{center}
\end{figure}

The corresponding almost-Riordan array 
\[
R=\left(d(t)|g(t),f(t)\right)=\left[\begin{array}{ccccccc}
1&&&&&&\\
1&1&&&&&\\
\frac 43&2&1&&&&\\
2&\frac{10}3&5&1&&&\\
\frac{28}9&\frac{16}3&\frac{58}3&8&1&&\\
\frac{44}9&\frac{76}9&\frac{214}3&\frac{133}3&11&1&\\
\vdots&&&&&&\ddots\\
\end{array}\right]
\]
with the characteristic sequences $AZW_4(3,0)$ and 
\[
\begin{array}{l}
d(t)=\frac{3-3t}{2t^2-6t+3},\\
g(t)=\frac 3{2t^2-6t+3},\\
f(t)=\frac{1-3t-\sqrt{5t^2-6t+1}}{2t},\\
\end{array}
\]
is TP, as well as the array 
\[
R=\left(d(t)|g(t),f(t)\right)=\left[\begin{array}{ccccccc}
1&&&&&&\\
1&1&&&&&\\
\frac{17}{12}&2&1&&&&\\
\frac{19}8&\frac{27}8&\frac 92&1&&&\\
\frac{425}{96}&\frac{89}{16}&\frac{125}8&7&1&&\\
\frac{297}{32}&\frac{1793}{192}&\frac{413}8&\frac{273}8&\frac{19}2&1&\\
\vdots&&&&&&\ddots\\
\end{array}\right]
\]
with the characteristic sequences $AZW_4\left(2\frac 12,\frac 1{24}\right)$ and
\[
\begin{array}{l}
d(t)=\frac{101t-98+\sqrt{9t^2-20t+4}}{(t-1)\sqrt{9t^2-20t+4}-59t^2+185t-94},\\
g(t)=\frac{96}{(1-t)\sqrt{9t^2-20t+4}+59t^2-185t+94},\\
f(t)=\frac{2-5t-\sqrt{9t^2-20t+4}}{4t}.\\
\end{array}
\]
\end{example}

\subsection{Counterexample}

Proposition 3.1 raises a natural question: does the converse theorem also hold? Specifically, is the production matrix $J$ TP for every TP almost-Riordan array 
$R$? According to \cite{Slo}, since the direct sum of any Riordan array with the matrix $\left[ 1\right]$ results in an almost-Riordan array, the answer is negative. Nevertheless, we will present more example in our discussion.

\begin{example}
Since $f(t)=2t+t^2$ is a P\'olya frequency formal power series, by Theorem \ref{thm:2.4}, for every P\'olya frequency formal power series $g(t)$ and any $d(t)=d_0+d_1t$ with $d_0,d_1>0$, the almost-Riordan array $\left(d|g,f\right)$ is TP. Yet, since $\overline{f}(t)=\sqrt{t+1}-1$, we have 
\[
A(t)=\sqrt{t+1}+1=2+\frac t2-\frac{t^2}8+\frac{t^3}{16}-\frac{5t^4}{128}+\frac{7t^5}{256}-\cdots
\]
Since $A$ has some negative coefficients, the production matrix $J$ is, definitely, not TP.
\par 
For instance 
\begin{align*}
R=&(d(t)|\, g(t), f(t))=\left(1+3t|\,1+t,2t+t^2\right)\\
=&\left[\begin{matrix}
1&&&&&&\\
3&1&&&&&\\
0&1&2&&&&\\
0&0&3&4&&&\\
0&0&1&10&8&&\\
0&0&0&5&24&16&\\
\vdots&&&&&&\ddots\\
\end{matrix}\right]
\end{align*}
is TP, but its production matrix generated by 
\begin{align*}
&A(t)=\sqrt{t+1}+1,\\
&Z(t)=\frac{\sqrt{t+1}-2t}{\sqrt{t+1}},\quad z_0=g_0/d_0=1\\
&W(t)=\frac{3\sqrt{t+1}-9t}{\sqrt{t+1}}, \quad w_0=d_1/d_0=3,
\end{align*}
begins 
\[
J=\left[\begin{matrix}
3&1&&&&&\\
-9&-2&2&&&&\\
\frac{9}{2}&1&\frac 12&2&&&\\
-\frac{27}{8}&-\frac{3}{4}&-\frac 18&\frac 12&2&&\\
\frac{45}{16}&\frac{5}{8}&\frac 1{16}&-\frac 18&\frac 12&2&\\
\vdots&&&&&&\ddots\\
\end{matrix}\right],
\]
which is not TP. 
\end{example}

\end{document}